\documentclass[11pt]{article}
\usepackage{amsfonts,amsthm, amsmath}
\usepackage{epsfig}
\usepackage{makeidx}
\usepackage{graphicx}
\usepackage{graphicx,epstopdf}
\DeclareGraphicsExtensions{.ps,.png,.jpg}
\usepackage{enumerate}

\makeindex
\usepackage{color}

\newcommand{\ncom}{\newcommand}
\ncom{\ul}{\underline}
\ncom{\beq}{\begin{equation}}
\ncom{\eeq}{\end{equation}}
\ncom{\bea}{\begin{eqnarray*}}
\ncom{\eea}{\end{eqnarray*}}
\ncom{\beqa}{\begin{eqnarray}}
\ncom{\eeqa}{\end{eqnarray}}
\ncom{\nno}{\nonumber}
\ncom{\non}{\nonumber}
\ncom{\ds}{\displaystyle}
\ncom{\half}{\frac{1}{2}}
\ncom{\mbx}{\makebox{.25cm}}
\ncom{\hs}{\mbox{\hspace{.25cm}}}
\ncom{\rar}{\rightarrow}
\ncom{\Rar}{\Rightarrow}
\ncom{\noin}{\noindent}
\ncom{\bc}{\begin{center}}
\ncom{\ec}{\end{center}}
\ncom{\sz}{\scriptsize}
\ncom{\rf}{\ref}
\ncom{\s}{\sqrt{2}}
\ncom{\sgm}{\sigma}
\ncom{\Sgm}{\Sigma}
\ncom{\psgm}{\sigma^{\prime}}
\ncom{\dt}{\delta}
\ncom{\Dt}{\Delta}
\ncom{\lmd}{s}
\ncom{\Lmd}{s}
\ncom{\Th}{\Theta}
\ncom{\e}{\eta}
\ncom{\eps}{\epsilon}
\ncom{\pcc}{\stackrel{P}{>}}
\ncom{\lp}{\stackrel{L_{p}}{>}}
\ncom{\dist}{{\rm\,dist}}
\ncom{\sspan}{{\rm\,span}}
\ncom{\re}{{\rm Re\,}}
\ncom{\im}{{\rm Im\,}}
\ncom{\sgn}{{\rm sgn\,}}
\ncom{\ba}{\begin{array}}
\ncom{\ea}{\end{array}}
\ncom{\hone}{\mbox{\hspace{1em}}}
\ncom{\htwo}{\mbox{\hspace{2em}}}
\ncom{\hthree}{\mbox{\hspace{3em}}}
\ncom{\hfour}{\mbox{\hspace{4em}}}
\ncom{\vone}{\vskip 2ex}
\ncom{\vtwo}{\vskip 4ex}
\ncom{\vonee}{\vskip 1.5ex}
\ncom{\vthree}{\vskip 6ex}
\ncom{\vfour}{\vspace*{8ex}}
\ncom{\norm}{\|\;\;\|}
\ncom{\integ}[4]{\int_{#1}^{#2}\,{#3}\,d{#4}}
\ncom{\vspan}[1]{{{\rm\,span}\{ #1 \}}}
\ncom{\dm}[1]{ {\displaystyle{#1} } }
\ncom{\ri}[1]{{#1} \index{#1}}

\newtheorem{theorem}{\bf Theorem}[section]

\newtheorem{proposition}{Proposition}[section]

\newtheorem{corollary}{Corollary}[section]

\newtheorem{definition}{Definition}[section]
\newtheoremstyle
    {remarkstyle}
    {}
    {11pt}
    {}
    {}
    {\bfseries}
    {:}
    {     }
    {\thmname{#1} \thmnumber{#2} }

\theoremstyle{remarkstyle}

\begin{document}

\newpage

\begin{center}
{\Large \bf Potential Theory of Normal Tempered Stable Process}
\end{center}
\vone
\begin{center}
{Arun Kumar}$^{\textrm{a}}$ and {Harsh Verma}$^{\textrm{b}}$

{\footnotesize{
		$$\begin{tabular}{l}
		$^{\textrm{a}}$ \emph{Department of Mathematics, Indian Institute of Technology Ropar, Rupnagar, Punjab 140001, India}\\
		
$^{b}$Department of Mathematical Sciences, Indian Institute of Science Education and Research Mohali,\\ Sector 81, SAS Nagar, Punjab 140306, India
		\end{tabular}$$}}
\end{center}

\vtwo
\begin{center}
\noindent{\bf Abstract}
\end{center}
In this article, we study the potential theory of normal tempered stable process which is obtained by time-changing the Brownian motion with a tempered stable subordinator. Precisely, we study the asymptotic behavior of potential density and L\'evy density associated with tempered stable subordinator and the Green function and the L\'evy density associated with the normal tempered stable process. We also provide the corresponding results for normal inverse Gaussian process which is a well studied process in literature.\\
\vone \noindent{\it Key words:} Tempered stable subordinator, L\'evy density, potential density, Green function. 
\vtwo
\setcounter{equation}{0}

\section{Introduction}

In recent years subordinated stochastic processes are getting increasing attention due to their wide applications in finance, statistical physics and biology. Subordinated stochastic processes are obtained by changing the time of some stochastic process called as parent or outer process by some other non-decreasing L\'evy process called subordinator or the inner process. Note that subordination is a convenient way to develop stochastic models where it is required to retain some properties of the parent process and at the same time it is required to alter some characteristics. Some well known subordinated processes include variance gamma process (Madan et al. 1998), normal inverse Gaussian process (Barndorff-Nielsen, 1998), fractional Laplace motion (Kozubowski et al. 2006), multifractal models (Mandelbrot et al. 1997), Student processes (Heyde and Leonenko, 2005), time-fractional Poisson process (see e.g. Laskin, 2003; Mainardi et al. 2004; Meerschaert et al. 2011), space-fractional Poisson process (Orsingher and Polito, 2012) and tempered space-fractional Poisson process (Gupta et al. 2020)  etc.  \\

\noindent The term ``potential theory" has its origin in the physics of $19$th century, where a dominant belief was that the fundamental forces of nature were to be derived from potentials which satisfied Laplace's equation (see e.g. Helms, 2009). This theory has its origin in the two well known theories of physics namely Gravitational and Electromagnetic theory. The term potential function was first associated with the work done in moving a point charge from one point of space to other in the presence of another external charge particle. The basic potential function varies as $\frac{1}{d}$, where $d$ is the distance between the particles and the dimension of the space is greater than or equal to 3. This function has a property that it satisfies Laplace equation and such functions are called harmonic functions. From a mathematical point of view, potential theory is basically the study of harmonic functions (Armitage and Gardiner,  2001). Potential theory has intimate connection with probability theory (see e.g. Doob, 2001). One important connection is that the transition function of a Markov process can be used to define the Green function related to the potential theory. In this paper, we study the potential theory of normal tempered stable process and nornal inverse Gaussian process by realizing them as Brownian motion time-changed tempered stable subordiantor and inverse Gaussian process respectively. Precisely, we study the asymptotic behavior of potential density and L\'evy density associated with different subordinators and also Green function and L\'evy density associated with the subordinate processes. In easy language, potential measure represents the average time stay of a subordinator in a Borel subset of real numbers and L\'evy measure quantifies the density of the number of jumps per unit time of the subordinator. The Green function denoted by $G(x,y) = G(x-y)$ for the Markov process is the expected amount of time spent at $y$ by the Markov process started at $x$ (see e.g. Doob, 2001;  Liggett, 2010).  The potential measure may be of interest to an investor, who is concern for the average time the stock prices stay in a particular price range.\\

\noindent In this paper, we will find the asymptotic behavior of potential density and L\'evy density associated with the tempered stable subordinator and the inverse Gaussian subordinator. Further, we will find the Green function and L\'evy density associated with the Brownian motion directed with these subordinators. More precisely, we find the asymptoric behaviour of Green function and L\'evy density associated with normal tempered stable process and nornal inverse Gaussian process.

\section{Preliminaries}
In this section, we state some important theorems and definitions which will be used later to prove other results. We start with Tauberian theorems. For more details and proof of these theorems please refer to  Theorems 1.7.1, 1.7.2, $1.7.1^\prime$ of Bingham et al. (1987). Tauberian theorems are used to find the asymptotic behaviour at infinity and zero of a real valued non-decreasing function by using the asymptotic behaviours of its Laplace-Stieltjes transform. If $f$ and $g$ are two function then $f\sim g$ means $\frac{f}{g}$ converges asymptotically to $1$.

\begin{definition}[Slowly varying function]
Let $l$ be a positive measurable function, defined on some neighbourhood $[x,\infty)$ of infinity, and satisfying $\frac{l(\lambda{x})}{l(x)}\rightarrow{1}$ as $(x\rightarrow{\infty})$  ${\forall \; \lambda>0}$; then $l$ is said to be slowly varying at infinity (for details see Bingham et al. 1987, p.6).
\end{definition}



\begin{theorem} (Karamata's Tauberian theorem)\label{Karamata-Tauberian} Let $U$ be a non-decreasing right continuous function on $\mathbb{R}$ with $U(x)=0$ $\forall$ $x<0.$ If $l$ slowly varies and $c\geq 0,$ $\rho{\geq 0}$, the following are equivalent:
\begin{enumerate}[a.]
\item As $x\rightarrow{\infty}$ (respectively $x\rightarrow 0+$),\; $U(x)\sim\frac{{c{x}^{\rho}l(x)}}{\Gamma(1+\rho)},$
\item As $s\rightarrow{0+}$ (respectively $x\rightarrow \infty$),\; $\bar{U}(s)\sim{c}{s}^{-\rho}l\left(\frac{1}{s}\right).$
\end{enumerate}
\end{theorem}

\begin{theorem} (Karamata's monotone density theorem)\label{Karamata-monotone-density} Let $U(x)=\int_{0}^{x}u(y)dy.$ If $U(x)\sim cx^{\rho}l(x)$ as  $x\rightarrow{\infty},$ where $c\in\mathbb{R},$ $\rho \in \mathbb{R}$, $l$ be a slowly varying function, and $u$ is ultimately monotone, then $u(x)\sim c\rho{x}^{\rho-1}l(x).$
\end{theorem}
\begin{definition}[Subordinator]
A subordinator is a 1-dimensional non-decreasing L\'evy process. A subordinator $S=(S(t):\; t\geq0)$ is characterized by its Laplace transform 
$${\mathbb{E}\left[\left(e^{- s S(t)}\right)\right]}=\exp\left({-t\phi(s)}\right).$$ The function $\phi$ is called the Laplace exponent of the subordinator (see Applebaum, 2009).
\end{definition}

\begin{definition}(Potential measure and density)
The potential measure of a subordinator $S(t)$ is defined by
\begin{equation*}
U(A)=\mathbb{E}\left[\int_{0}^{\infty}{1}_{({S(t)}\in{A})}dt\right],
\end{equation*}
where $A$ is a Borel subset of $(0,\infty).$ The potential measure has a density which is called the potential density of the subordinator (see e.g. Sikic  et al. 2006).
\end{definition}
\noindent Observe $U(A)$ is the expected time the subordinator $S(t)$ spends in the set $A$. Measure $U$ has the following Laplace-Stieltjes transform
\begin{align}\label{LT-Potential-Measure}
\bar{U}(s) = \int_{0}^{\infty}e^{-st}dU(t) =\mathbb{E} \int_{0}^{\infty} e^{-s S(t)} dt = \int_{0}^{\infty}e^{-t\phi(s)}dt= \frac{1}{\phi(s)}.
\end{align}
If the Laplace exponent $\phi$ is a complete Bernstein function, then potential measure has a density $u$ and $u$ is completely monotone. If the subordinator has a potential density $u(t)$ then, its Laplace transfrom is given by
\begin{equation}
    \bar{u}(t) = \int_{0}^{\infty}e^{-st}u(t) dt = \frac{1}{\phi(s)}.
\end{equation}
 If the L\'evy measure $\mu$ of $\phi$ has a completely monotone density $\mu(t)$, i.e., $(-1)^n D^n \mu\geq 0$ for every non-negative integer $n\geq 1$, which is equivalent to saying that $\phi$ is a complete Bernstein function.
\noindent Let $B(t)$ be a $d$-dimensional Brownian motion with transitional density given by
\begin{equation}\label{BM-transition}
{p(t,x,y)=\left(4\pi{t}\right)^{-\frac{d}{2}}\exp\left({-\frac{{\lvert{x-y}\rvert}^2}{4t}}\right),\; x,y\in \mathbb{R}^d,\; t>0}.  
\end{equation}

\begin{definition}(Green function for Markov process)
The Green function for the Markov process is defined by
$$
G(x,y) = \int_{0}^{\infty}p(t, x, y)dt,
$$
where $p(t,x,y)$ is the transition function of the Markov process. It is the expected amount of time spent at $y$ by the process started at $x$ (Liggett, 2010).
\end{definition}

\noindent Let $B(t)$ be $d$-dimensional Brownain motion and $S(t)$ be a subordinator then the process defined by $X(t) = B(S(t))$ is called a subordinated process to Brownian motion . The Green function of $X$ is given by (Sikic  et al. 2006)
\begin{equation}\label{Def-Green-function}
G(x) = \int_{0}^{\infty}p(t,0,x) u(t)dt,   
\end{equation}
where $u(t)$ is the potential density of the subordinator $S$ and $p(t,0,x)$ is the transition density of Brownian motion. Further, the L\'evy density of $X$ is given by (Sikic  et al. 2006)
\begin{equation}\label{Def-Levy-density}
J(x) = \int_{0}^{\infty}p(t,0,x)\mu(t)dt,   
\end{equation}
where $\mu(t)$ is the L\'evy density of $S$. 

\section{Tempered stable subordinator}
Let $S_{\lambda, \beta}(t)$ be the tempered stable subordinator (TSS) with index $\beta \in(0,1)$ and  tempering parameter $\lambda > 0$. Note that, TSS are obtained by exponential tempering in the distribution of stable subordinator (see e.g. Rosinski, 2007). The marginal pdf for TSS $S_{\lambda, \beta}(t)$ is given by
\beq
\label{ts-density}
g_{\lambda, \beta}(x,t)= e^{-\lambda x+\lambda^{\beta}t} f_{\beta}(x,t),~~ s>0, \;\beta\in (0,1), 
\eeq
where $f_{\beta}(x,t)$ is the pdf of a stable subordinator. The advantage of tempered stable process over a stable process is that its all moments exist and its density is also infinitely divisible. However, the process is not self-similar.
The Laplace transform of the TSS is given by 
\begin{equation*}
\mathbb{E}(e^{- s S_{\lambda, \beta}(t)})= e^{-t\left(\left(s + \lambda\right)^{\beta}-\lambda^{\beta}\right)}.
\end{equation*}
The L\'evy density associated with the tempered stable subordinator is given by (see e.g. Rosi\'nski, 2007; Kumar and Vellaisamy, 2015)
\begin{equation}\label{TSS-Levy-density}
\mu(x)=\frac{c e^{-\beta x}}{x^{\lambda+1}},\;x>0,
\end{equation}
where $c = \beta/\Gamma(1-\beta).$ Using direct computation, we can show that $(-1)^n D^n\mu \geq 0,\;n\geq 1$ for TSS. Hence the corresponding potential measure has a potential density.
\noindent We next discuss the asymptotic behavior of the potential density of the TSS.
\begin{theorem}
For $\beta\in(0,1)$ and $\lambda>0$, the potential density $u(x)$ for TSS have following asymptotic forms
\begin{enumerate}
\item $u(x)\sim \frac{\lambda x^{\lambda-1}}{\Gamma(1+\lambda)},\; {\rm as}\; x\rightarrow{0^+},$
\item $u(x)\sim\frac{1}{\lambda \beta^{\lambda-1}},\; {\rm as}\; x\rightarrow{\infty}.$
\end{enumerate}
\end{theorem}
\begin{proof}
\noindent\begin{enumerate}
\item The Laplace exponent of the Tempered stable subordinator is given by
\begin{equation*}
\phi(s)=\left(s+\beta\right)^{\lambda}-\beta^{\lambda},
\end{equation*}
where $\lambda\in(0,1)$ and $\beta>0.$ Using \eqref{LT-Potential-Measure}, the Laplace transform of the potential measure $U$ of tempered stable subordinator will be 
\begin{equation*}
\bar{U}(s)=\frac{1}{\phi(s)}=\frac{1}{\left(s+\beta\right)^{\lambda}-\beta^{\lambda}} \sim\frac{1}{\lambda\beta^{\lambda-1}s},\; s\rightarrow{0^+}.
\end{equation*}



Therefore, by Theorem \ref{Karamata-Tauberian}, the potential measure is
\begin{equation*}
U(x)\sim\frac{1}{\lambda\beta^{\lambda-1}\Gamma(2)}x,\; x\rightarrow{\infty}.
\end{equation*}

Hence, by Theorem \ref{Karamata-monotone-density}, the potential density
\begin{equation*}
u(x)\sim\frac{1}{\lambda\beta^{\lambda-1}\Gamma(2)} = \frac{1}{\lambda\beta^{\lambda-1}},\; x\rightarrow{\infty}.
\end{equation*}
\item As $s\rightarrow \infty$, we have $\phi(s)\sim s^{\lambda}$, therefore, $\bar{U}(s)\sim\frac{1}{s^{\lambda}}.$ Thus, by Theorem \ref{Karamata-Tauberian}, we have
\begin{equation*}
U(x)\sim\frac{x^\lambda}{\Gamma(1+\lambda)},\;x\rightarrow{0^+}.
\end{equation*}
Using Theorem \ref{Karamata-monotone-density}, 
\begin{equation*}
u(x)\sim \frac{\lambda x^{\lambda-1}}{\Gamma(1+\lambda)},\; x\rightarrow{0^+}.
\end{equation*}
\end{enumerate}
\end{proof}

\noindent In next result, we find the exact form of the potential density.
\begin{theorem}
For $\lambda\in(0,1)$ and $\beta>0$, we have
$$u(x)=\exp(-\beta x)\frac{\sin(\pi\lambda)}{\pi}\int_{0}^{\infty}\frac{\exp(-xu)u^\lambda}{u^{2\lambda}+\beta^{2\lambda}-2u^\lambda\beta^{\lambda}\cos(\pi\lambda)}du.$$
\end{theorem}
\begin{proof}
Let $\phi(s) = f(s+\beta)=(s+\beta)^\lambda-\beta^\lambda$ be the Laplace exponent of the tempered stable subordinator, where  $f(s)=s^\lambda-\beta^{\lambda}.$ Then using the properties of inverse Laplace transform we have 
\begin{equation*}
\mathcal{L}^{-1}\left[\bar{u}(s)=\frac{1}{\phi(s)}\right](t)=\exp(-\beta x)\mathcal{L}^{-1}\left[\frac{1}{f(s)}\right](x),
\end{equation*}
where $\mathcal{L}^{-1}[\bar{u}(s)]$ is the inverse Laplace transform of the potential density related to the tempered stable subordinator.
\begin{figure}[ht!]
    \centering
    \includegraphics[scale=0.8]{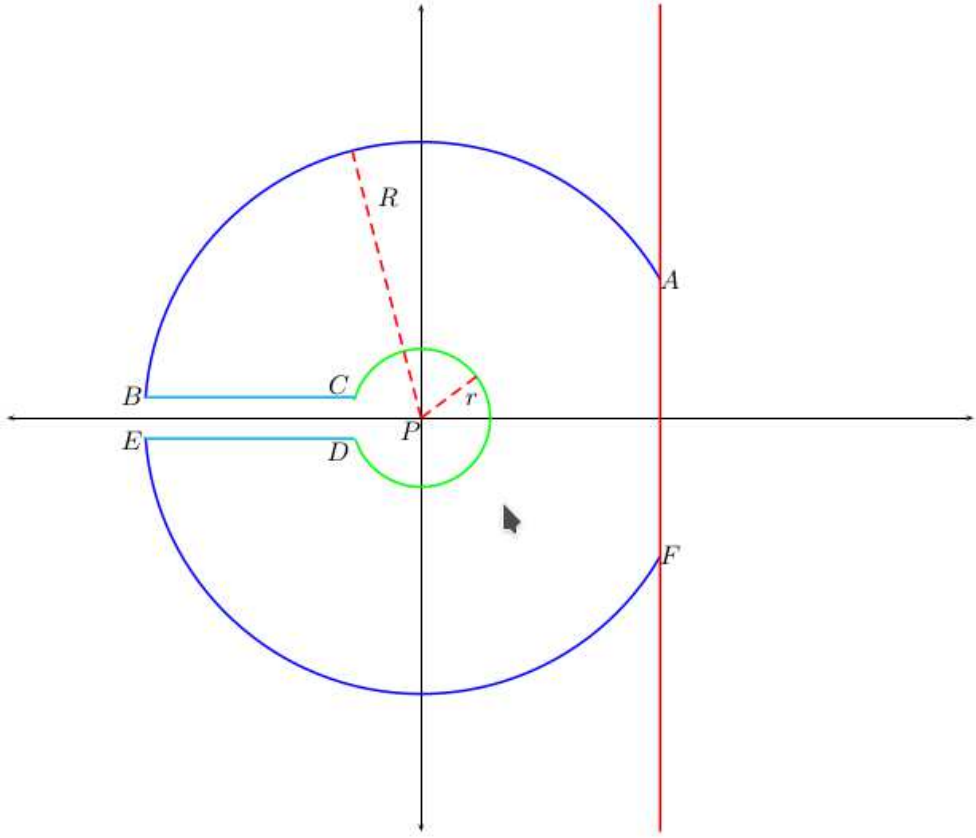}
    \caption{Contour C anti-clockwise}
    \label{fig:contour}
\end{figure}
Now to find the inverse Laplace transform of the potential density we will first find $\mathcal{L}^{-1}\left[\frac{1}{f(s)}\right](x)$ using complex inversion formula (Schiff, 1999). Since $s=0$ is a branch point of $\frac{1}{f(s)}$, so we take a branch cut along non-positive real line to make the function single valued as shown in the Figure  \eqref{fig:contour}. Inside and on the contour the function is analytic so by Cauchy's theorem 
\begin{equation*}
\frac{1}{2\pi i}\int_{C}^{}\frac{\exp(s x)}{f(s)}ds=0.
\end{equation*}
Now 
\begin{align*}
\frac{1}{2\pi i}\int_{C}^{}\frac{\exp(sx)}{f(s)}ds & =
\frac{1}{2\pi i}\int_{AB}^{}\frac{\exp(sx)}{f(s)}ds+\frac{1}{2\pi i}\int_{BC}^{}\frac{\exp(sx)}{f(s)}ds+\frac{1}{2\pi i}\int_{CD}^{}\frac{\exp(sx)}{f(s)}ds\\
&+\frac{1}{2\pi i}\int_{DE}^{}\frac{\exp(sx)}{f(s)}ds+\frac{1}{2\pi i}\int_{EF}^{}\frac{\exp(sx)}{f(s)}ds+\frac{1}{2\pi i}\int_{FA}^{}\frac{\exp(sx)}{f(s)}ds=0.
\end{align*}
also 
\begin{equation*}
\int_{AB}^{}\frac{\exp(sx)}{f(s)}ds=\int_{CD}^{}\frac{\exp(sx)}{f(s)}ds=\int_{EF}^{}\frac{\exp(sx)}{f(s)}ds=0,
\end{equation*}
(see Schiff, 1999 for details). We know from  Schiff (1999) that as $r\rightarrow{0}$ and $R\rightarrow{\infty}$
\begin{equation}
\frac{1}{2\pi i}\int_{FA}^{}\frac{\exp(sx)}{f(s)}ds=L^{-1}\left[\frac{1}{f(s)}\right](t),
\end{equation}
therefore,
\begin{equation}\label{interm1}
\mathcal{L}^{-1}\left[\frac{1}{f(s)}\right](x)=-\left(\frac{1}{2\pi i}\int_{BC}^{}\frac{\exp(sx)}{f(s)}ds+\frac{1}{2\pi i}\int_{DE}^{}\frac{\exp(sx)}{f(s)}ds\right).
\end{equation}
Consider, 
\begin{equation*} 
\int_{BC}^{}\frac{\exp(sx)}{f(s)}ds=\int_{BC}^{}\frac{\exp(sx)}{s^\lambda-\beta^\lambda}ds=\int_{-R}^{-r}-\frac{\exp(sx)}{s^\lambda-\beta^\lambda}ds.
\end{equation*}
Let $s = u e^{i\pi}$ then $ds = -du$, thus,
\begin{equation*}
\int_{BC}^{}\frac{\exp(s x)}{s^\lambda-\beta^\lambda}ds=\int_{R}^{r}-\frac{\exp(- ux)}{u^\lambda e^{i\pi \lambda}-\beta^\lambda}du.
\end{equation*}
Taking the limits $r\rightarrow{0}$ and $R\rightarrow{\infty}$ on both side of the above equation we get
\begin{equation}\label{interm2}
\int_{BC}^{}\frac{\exp(sx)}{s^\lambda-\beta^\lambda}ds=\int_{0}^{\infty}\frac{\exp(-ux)}{u^\lambda\exp(i\pi\lambda)-\beta^\lambda}du.
\end{equation}
Now
\begin{equation*}
\int_{DE}^{}\frac{\exp(sx)}{f(s)}ds=\int_{DE}^{}\frac{\exp(sx)}{s^\lambda-\beta^\lambda}ds=\int_{-r}^{-R}-\frac{\exp(sx)}{s^\lambda-\beta^\lambda}ds,
\end{equation*}
For DE, take $s=ue^{-i\pi}$, which gives $ds=-du$,
\begin{equation*}
\int_{DE}^{}\frac{\exp(sx)}{s^\lambda-\beta^\lambda}ds=\int_{r}^{R}-\frac{\exp(-xu)}{u^\lambda e^{-i\pi\lambda}-\beta^\lambda}du.
\end{equation*}
Again we take the limits $r\rightarrow{0}$ and $R\rightarrow{\infty}$ on both side of the above equation we get 
\begin{equation}\label{interm3}
\int_{DE}^{}\frac{\exp(sx)}{s^\lambda-\beta^\lambda}ds=\int_{0}^{\infty}-\frac{\exp(-xu)}{u^\lambda\exp(-i\pi\lambda)-\beta^\lambda}du.
\end{equation}
From equations \eqref{interm1}, \eqref{interm2} and \eqref{interm3}, we get
\begin{align*}
\mathcal{L}^{-1}\left[\frac{1}{f(s)}\right](x)& =-\frac{1}{2\pi i}\left[\int_{0}^{\infty}\frac{\exp(-xu)}{u^\lambda\exp(i\pi\lambda)-\beta^\lambda}du - \int_{0}^{\infty}\frac{\exp(-xu)}{u^\lambda\exp(-i\pi\lambda)-\beta^\lambda}du\right]\\
&=-\frac{1}{2\pi i}\int_{0}^{\infty}\exp(-xu)\left[\frac{1}{u^\lambda\exp(i\pi\lambda)-\beta^\lambda}-\frac{1}{u^\lambda\exp(-i\pi\lambda)-\beta^\lambda}\right]du\\
&=-\frac{1}{2\pi i}\int_{0}^{\infty}\exp(-xu)\frac{-2i\sin(\pi\lambda)u^\lambda}{u^{2\lambda}+\beta^{2\lambda}-2\cos(\pi\lambda)u^\lambda}du\\
& =\frac{\sin(\pi\lambda)}{\pi}\int_{0}^{\infty}\frac{\exp(-xu)u^\lambda}{u^{2\lambda}+\beta^{2\lambda}-2u^{\lambda}\beta^{\lambda}\cos(\pi\lambda)}du
\end{align*}
Thus, the potential density is
\begin{equation*}
u(x) = \mathcal{L}^{-1}[\bar{u}(s)](x) = \exp(-\beta x)\frac{\sin(\pi\lambda)}{\pi}\int_{0}^{\infty}\frac{\exp(-xu)u^\lambda}{u^{2\lambda}+\beta^{2\lambda}-2u^\lambda\beta^{\lambda}\cos(\pi\lambda)}du.
\end{equation*}
\end{proof}

\section{Normal tempered stable process}

In this section, we compute the asymptotic behavior of the Green function and L\'evy density associated with the Brownian motion time-changed by tempered stable subordinator which is called the normal tempered stable process see Barndorff-Nielsen and Shephard (2001). Let $B(t)$ be the Brownian motion with transition probability given in \eqref{BM-transition} and $S_{\lambda,\beta}(t)$ be the TSS, then the time change process $X(t) = B(S_{\lambda,\beta}(t))$ is called normal tempered stable process.
\begin{proposition}
For $\lambda\in(0,1)$ and $\beta>0$, we have
\begin{enumerate}
\item $G(x)\sim\frac{\Gamma\left(\frac{d-2\lambda}{2}\right)}{\pi^{\frac{d}{2}}4^\lambda \Gamma(\lambda)}\lvert{x}\rvert^{2\lambda-d},\;\lvert{x}\rvert\rightarrow{0^+}$,
\item $G(x)\sim\frac{\Gamma\left(\frac{d-2}{2}\right)}{4\pi^{\frac{d}{2}}\lambda\beta^{\lambda-1}}\lvert{x}\rvert^{2-d},\;\lvert{x}\rvert\rightarrow{\infty}$.
\end{enumerate}
\end{proposition}
\begin{proof}
\noindent\begin{enumerate}
\item As $s\rightarrow{\infty}$, $\phi(s)\sim s^{\lambda}$, then it follows directly from Theorem 3.1 of Rao et al. (2006) that 
\begin{equation*}
G(x)\sim\frac{\Gamma\left(\frac{d-2\lambda}{2}\right)}{\pi^{\frac{d}{2}}4^\lambda \Gamma(\lambda)}\lvert{x}\rvert^{2\lambda-d},\; \lvert{x}\rvert\rightarrow{0^+}.
\end{equation*}
\item As $s\rightarrow{0^+}$, $\phi(s)\sim\lambda\beta^{\lambda-1}s$, then it follows directly from Theorem 3.3 of Rao et al. (2006) that 
\begin{equation*}
G(x)\sim\frac{1}{\pi^{\frac{d}{2}}2^2\lambda\beta^{\lambda-1}}\frac{\Gamma\left(\frac{d-2}{2}\right)}{\Gamma\left(\frac{2}{2}\right)}\lvert{x}\rvert^{2-d},\; \lvert{x}\rvert\rightarrow{\infty}
\end{equation*}
or 
\begin{equation*}
G(x)\sim\frac{\Gamma\left(\frac{d-2}{2}\right)}{4\pi^{\frac{d}{2}}\lambda\beta^{\lambda-1}}\lvert{x}\rvert^{2-d}, \; \lvert{x}\rvert\rightarrow{\infty}.
\end{equation*}
\end{enumerate}
\end{proof}
\noindent We now find the asymptotic behaviour of the L\'evy density $J(x)$ of the subordinated process. We can not use the same method as in Sikic et al. (2006) because the asymptotic behavior of L\'evy density $\mu(x)$ of the L\'evy measure $\mu$ associated with the tempered stable subordinator does not follow assumption 2 of Lemma 3.1 so we can't use the results given in Sikic et al. (2006). 
\begin{theorem} \label{TSS-Asymptotic-Levy}
For $\lambda\in(0,1)$ and $\beta>0$, we have
\begin{enumerate}
\item $J(x)\sim \frac{4^{\lambda+1}}{\pi^{\frac{d}{2}}}\left(\lambda+\frac{d}{2}\right)c\lvert{x}\rvert^{-(2\lambda+d)},\;\lvert{x}\rvert\rightarrow{0},$
\item $J(x)\sim\frac{2^{\frac{2\lambda-d+1}{2}}\beta^{\frac{2\lambda+d-1}{4}}}{\pi^{\frac{d-1}{2}}}c\lvert{x}\rvert^{-\frac{2\lambda+d-1}{2}}\exp\left(-\sqrt{\beta}\lvert{x}\rvert\right),\;\lvert{x}\rvert\rightarrow{\infty}.$
\end{enumerate}
\end{theorem}
\begin{proof}

Using \eqref{Def-Levy-density} and \eqref{TSS-Levy-density}, the L\'evy density associated with the subordinated process
\begin{align*}
J(x)&=\int_{0}^{\infty}p(t,0,x)\mu(t)dt=\int_{0}^{\infty}\left(4\pi{t}\right)^{-\frac{d}{2}}\exp{\left(-\frac{(\lvert{x}\rvert)^2}{4t}\right)}\frac{c\exp(-\beta t)}{t^{\lambda+1}}dt\\
&=\frac{c}{2^d\pi^\frac{d}{2}}\int_{0}^{\infty}t^{-\left(\lambda+1+\frac{d}{2}\right)}\exp\left(-\frac{(\lvert{x}\rvert)^2}{4t}-\beta t\right)dt\\
&=\frac{c}{2^d\pi^\frac{d}{2}}\int_{0}^{\infty}t^{-\left(\lambda+1+\frac{d}{2}\right)}\exp\left[-\frac{1}{2}\left(\frac{(\lvert{x}\rvert)^2}{2t}+2\beta t\right)\right]dt.
\end{align*}
Let $t=\frac{\lvert{x}\rvert y}{2\sqrt{\beta}},$ then $dt=\frac{\lvert{x}\rvert dy}{2\sqrt{\beta}},$ thus,
\begin{align}
J(x)&=\frac{c}{2^d\pi^\frac{d}{2}}\int_{0}^{\infty}\left[\frac{\lvert{x}\rvert y}{2\sqrt{\beta}}\right]^{-\left(\lambda+1+\frac{d}{2}\right)}\exp\left[-\frac{1}{2}\left(\sqrt{\beta}\lvert{x}\rvert y+\frac{\sqrt{\beta}\lvert{x}\rvert y}{y}\right)\right]\frac{\lvert{x}\rvert dy}{2\sqrt{\beta}}\nonumber\\
&=\frac{c}{2^d\pi^\frac{d}{2}}\left[\frac{\lvert{x}\rvert }{2\sqrt{\beta}}\right]^{-\left(\lambda+\frac{d}{2}\right)}\int_{0}^{\infty}y^{-\left(\lambda+1+\frac{d}{2}\right)}\exp\left[-\frac{1}{2}\sqrt{\beta}\lvert{x}\rvert\left(y+\frac{1}{y}\right)\right]dy\nonumber\\
&=\frac{2^{\left(\lambda-\frac{d}{2}\right)}\beta^{\frac{1}{2}\left(\lambda+\frac{d}{2}\right)}}{\pi^\frac{d}{2}}c\lvert{x}\rvert^{-\left(\lambda+\frac{d}{2}\right)}\int_{0}^{\infty}y^{-\left(\lambda+1+\frac{d}{2}\right)}\exp\left[-\frac{1}{2}\sqrt{\beta}\lvert{x}\rvert\left(y+\frac{1}{y}\right)\right]dy\nonumber\\
&=\frac{2^{\left(\lambda-\frac{d}{2}\right)}\beta^{\frac{1}{2}\left(\lambda+\frac{d}{2}\right)}}{\pi^\frac{d}{2}}c\lvert{x}\rvert^{-\left(\lambda+\frac{d}{2}\right)}2K_{-\nu}(\omega),
\end{align}
where 
\begin{align*}
K_{-\nu}(\omega)&=\frac{1}{2}\int_{0}^{\infty}y^{-\left(\lambda+1+\frac{d}{2}\right)}\exp\left[-\frac{1}{2}\sqrt{\beta}\lvert{x}\rvert\left(y+\frac{1}{y}\right)\right]dy,
\end{align*}
is a modified Bessel function of third type with
$\nu=\left(\lambda+\frac{d}{2}\right)$ and $\omega=\sqrt{\beta}\lvert{x}\rvert.$
Further,  $K_{- \nu}(\omega)=K_{\nu}(\omega)$ (Jorgensen, 1982).
\begin{enumerate}
\item As $\lvert{x}\rvert\rightarrow{0}$, $\omega\rightarrow{0}$ thus, using $A.7$ of Jorgensen (1982), we get
\begin{equation*}
K_{\nu}(\omega)\sim\lambda\left(\lambda+\frac{d}{2}\right)2^{\lambda+\frac{d}{2}-1}(\sqrt{\beta}\lvert{x}\rvert)^{-\left(\lambda+\frac{d}{2}\right)}.
\end{equation*}
We get
\begin{align*}
J(x) & \sim\frac{2^{\left(\lambda-\frac{d}{2}\right)}\beta^{\frac{1}{2}\left(\lambda+\frac{d}{2}\right)}}{\pi^\frac{d}{2}}c\lvert{x}\rvert^{-\left(\lambda+\frac{d}{2}\right)}2\lambda\left(\lambda+\frac{d}{2}\right)2^{\lambda+\frac{d}{2}+1}(\sqrt{\beta}\lvert{x}\rvert)^{-\left(\lambda+\frac{d}{2}\right)}\\
&\sim \frac{4^{\lambda+1}}{\pi^{\frac{d}{2}}}\left(\lambda+\frac{d}{2}\right)c\lvert{x}\rvert^{-(2\lambda+d)},\;\lvert{x}\rvert\rightarrow{0}.
\end{align*}
\item As $\lvert{x}\rvert\rightarrow \infty$, $\omega\rightarrow \infty$, thus from Jorgensen (1982), we get 
\begin{equation*}
K_{\nu}(\omega)\sim\sqrt{\frac{\pi}{2}}\exp\left(-\sqrt{\beta}\lvert{x}\rvert\right)\left(\sqrt{\beta}\lvert{x}\rvert\right)^{-\frac{1}{2}}.
\end{equation*}
Thus,
\begin{align*}
J(x) & \sim\frac{2^{\left(\lambda-\frac{d}{2}\right)}\beta^{\frac{1}{2}\left(\lambda+\frac{d}{2}\right)}}{\pi^\frac{d}{2}}c\lvert{x}\rvert^{-\left(\lambda+\frac{d}{2}\right)}2\sqrt{\frac{\pi}{2}}\exp\left(-\sqrt{\beta}\lvert{x}\rvert\right)\left(\sqrt{\beta}\lvert{x}\rvert\right)^{-\frac{1}{2}}\\
& \sim\frac{2^{\frac{2\lambda-d+1}{2}}\beta^{\frac{2\lambda+d-1}{4}}}{\pi^{\frac{d-1}{2}}}c\lvert{x}\rvert^{-\frac{2\lambda+d-1}{2}}\exp\left(-\sqrt{\beta}\lvert{x}\rvert\right),\; \lvert{x}\rvert \rightarrow{\infty}.
\end{align*}
\end{enumerate}
\end{proof}

\section{Inverse Gaussian Subordinator}
The PDF of a random variable following an inverse Gaussian (IG) distribution with parameters $\delta$ and $\lambda$ is given by Applebaum (2009)
\begin{equation*}
    g(x) = \frac{\delta e^{\delta\lambda}}{\sqrt{2\pi x^3}} e^{-\frac{1}{2}(\frac{\delta^2}{x} + \lambda^2 x)},\;x>0, \delta, \lambda >0.
\end{equation*}
IG distributions are infinitely divisible and the corresponding L\'evy process is called IG subordinator. Let $S(t)$ be the IG subordinator which has alternative representation 
\begin{equation}\label{IG-subordinator}
S(t) = \inf\{s>0 : B(s) + \lambda s = \delta t\},
\end{equation}
where $B(t)$ is the standard one-dimensional Brownian motion.
The Laplace transform of the inverse Gaussian subordinator is given by
\begin{equation*}
\mathbb{E}[e^{-s S(t)}]=\exp\left[-t \delta\left(\sqrt{2s+\lambda^2}-\lambda\right)\right].
\end{equation*}
The L\'evy density for IG subrodinator is
\begin{equation}
    \mu(x) = \frac{\delta}{\sqrt{2\pi x^3}} e^{-\frac{\lambda^2}{2}x},\;x>0. 
\end{equation}
\noindent Next, we discuss about the asymptotic behaviour of potential density of the IG subordinator.
\begin{theorem}\label{Th:Potential-IG}
For $\delta,\lambda>0$, we have
\begin{enumerate}
\item $u(x)\sim\frac{1}{\delta\sqrt{2\pi x}}, \; x\rightarrow{0^+}$,
\item $u(x)\sim\frac{\lambda}{\delta}, \;  x\rightarrow{\infty}$.
\end{enumerate}
\end{theorem}
\begin{proof}
\noindent\begin{enumerate}
\item The Laplace exponent of the Inverse Gaussian subordinator is given by
\begin{equation*}
\phi(s)=\delta\left(\sqrt{2s+\lambda^2}-\lambda\right),
\end{equation*}
where $\lambda\in(0,1)$ and $\beta>0.$ Therefore, the Laplace transform of the potential measure $U$ of tempered stable subordinator will be 
\begin{equation*}
\bar{U}(s)=\frac{1}{\delta\left(\sqrt{2s+\lambda^2}-\lambda\right)}.
\end{equation*}

Since $\delta\left(\sqrt{2s+\lambda^2}-\lambda\right)=\delta\lambda\left[\left(1+\frac{2s}{\lambda^2}\right)^{\frac{1}{2}}-1\right]$ and as $s\rightarrow{0^+}$ $\left(1+\frac{2s}{\lambda^2}\right)^{\frac{1}{2}}\sim1+\frac{s}{\lambda^2}$, therefore as $s\rightarrow{0^+}$, $\phi(s)\sim\delta\lambda\left[\left(1+\frac{s}{\lambda^2}\right)^{\frac{1}{2}}-1\right]=\frac{\delta}{\lambda}s,$ thus,
\begin{equation*}
\bar{U}(s)\sim\frac{\lambda}{\delta s},\; s\rightarrow{0^+}.
\end{equation*}

Therefore, by Theorem \ref{Karamata-Tauberian}, we have
\begin{equation*}
U(x)\sim\frac{\lambda x}{\delta\Gamma(2)} = \frac{\lambda x}{\delta},\; x\rightarrow{\infty},
\end{equation*}
and hence, by Theorem \ref{Karamata-monotone-density}, the potential density is given by
\begin{equation*}
u(x)\sim\frac{\lambda}{\delta},\; x\rightarrow{\infty}.
\end{equation*}
\item As $s\rightarrow \infty$, $\phi(s)\sim\delta\sqrt{2s}$, therefore, $\bar{U}(s)\sim\frac{1}{\delta\sqrt{2s}}.$ Thus, by Theorem \ref{Karamata-Tauberian}, we have
\begin{equation*}
U(x)\sim\frac{1}{\delta\sqrt{2}\Gamma\left(\frac{3}{2}\right)}\sqrt{x} = \frac{\sqrt{2x}}{\delta\sqrt{\pi}},\; x\rightarrow{0^+}.
\end{equation*}
Therefore, by Theorem \ref{Karamata-monotone-density}, we have 
\begin{equation*}
u(x)\sim\frac{1}{\delta\sqrt{2\pi x}}, \; x\rightarrow{0^+}.
\end{equation*}
\end{enumerate}
\end{proof}

\noindent Next we find the exact form of the potential density associated with the inverse Gaussian subordinator.
\begin{theorem}
For $\delta,\lambda>0$, the potential density of IG subordinator is given by
\begin{equation}\label{Potential-IG-Exact}
u(x)=\frac{1}{\sqrt{2}\delta}\left[\frac{\exp\left(-\frac{\lambda^2{x}}{2}\right)}{\sqrt{\pi x}} + \frac{\lambda}{\sqrt{2}}{\rm erfc}\left(-\frac{\lambda\sqrt{x}}{\sqrt{2}}\right)\right].
\end{equation}
\end{theorem}
\begin{proof}
Let $\phi(s)
= f(s+\frac{\lambda^2}{2})=\sqrt{2}\delta\left(\sqrt{s+\frac{\lambda^2}{2}}-\frac{\lambda}{\sqrt{2}} \right)$ be the Laplace exponent of the inverse Gaussian subordinator, where  $f(s)=\sqrt{2}\delta\left(\sqrt{s}-\frac{\lambda}{\sqrt{2}}\right).$ Then using the properties of inverse Laplace transform, we get
\begin{equation}\label{ILT1}
u(s) = \mathcal{L}^{-1}\left[\bar{u}(s)=\frac{1}{\phi(s)}\right](x)=\exp\left(-\frac{\lambda^2 x}{2}\right)\mathcal{L}^{-1}\left[\frac{1}{f(s)}\right](x).
\end{equation}

\noindent Using formula 128 on page 16 of Poularikas (1998) 
\begin{align}\label{ILT2}
\mathcal{L}^{-1}\left[\frac{1}{f(s)}\right](x)&= \mathcal{L}^{-1}\left[\frac{1}{\sqrt{2}\delta\left(\sqrt{s}-\frac{\lambda}{\sqrt{2}}\right)}\right](x)
=\frac{1}{\sqrt{2}\delta}\left[\frac{1}{\sqrt{\pi x}}+\frac{\lambda}{\sqrt{2}}\exp\left(\frac{\lambda^2 x}{2}\right){\rm erfc}\left(-\frac{\lambda\sqrt{x}}{\sqrt{2}}\right)\right] 
\end{align}
Using \eqref{ILT1} and \eqref{ILT2}, we get
\begin{equation*}
u(x)= \mathcal{L}^{-1}\left[\bar{u}(s)\right]=\frac{1}{\sqrt{2}\delta}\left[\frac{\exp\left(-\frac{\lambda^2{x}}{2}\right)}{\sqrt{\pi x}}+\frac{\lambda}{\sqrt{2}}{\rm erfc}\left(-\frac{\lambda\sqrt{x}}{\sqrt{2}}\right)\right].
\end{equation*}
\end{proof}

\begin{corollary}
The asymptotic behaviour in Theorem \ref{Th:Potential-IG} can be proved using \eqref{Potential-IG-Exact}.
\end{corollary}
\begin{proof}
Note that
\begin{align*}
  u(x) &=\frac{1}{\sqrt{2}\delta}\left[\frac{\exp\left(-\frac{\lambda^2{x}}{2}\right)}{\sqrt{\pi x}}+\frac{\lambda}{\sqrt{2}}{\rm erfc}\left(-\frac{\lambda\sqrt{x}}{\sqrt{2}}\right)\right]\\
  & = \frac{1}{\delta\sqrt{2\pi x}}\left[\exp\left(-\frac{\lambda^2{x}}{2}\right) + \frac{\lambda \sqrt{\pi x}}{\sqrt{2}}{\rm erfc}\left(-\frac{\lambda\sqrt{x}}{\sqrt{2}}\right) \right] \sim \frac{1}{\delta\sqrt{2\pi x}},\; x\rightarrow 0+,
\end{align*}
since $\mbox{erfc}(0) = 1.$ Further for large $x$, we have (see e.g. Abramowitz and Stegun, 1992)
\begin{equation}\label{erfc-expansion}
    {\rm erfc(x)} = \frac{e^{-x^2}}{x\sqrt{\pi}}\left[1+ \sum_{n=1}^{\infty}(-1)^n \frac{1\cdot 3 \cdot 5\cdots (2n-1)}{(2x^2)^n}\right].
\end{equation}
Using, $\mbox{erfc}(-x) = 2 - \mbox{erfc}(x)$ and \eqref{erfc-expansion}, we have
$$
u(x) \sim \frac{\lambda}{\delta},\; x\rightarrow \infty.
$$
\end{proof}

\noindent Next we compute the asymptotic behavior of Green function and L\'evy density associated with the normal inverse Gaussian process which is obtained by subordinating Brownian motion with inverse Gaussian subordinator.

\section{Normal Inverse Gaussian Process}
Let $B(t)$ be the $d$-dimensional Brownian motion with transition probability given in \eqref{BM-transition} and $S(t)$ be the IG subordinator defined in \eqref{IG-subordinator}. The process defined by
$$
Y(t) = B(S(t)),
$$
is called a $d$-dimensional normal inverse Gaussian process without drift (see e.g. Barndorff-Nielsen, 1998). The asymptotic form of  associated L\'evy density of this $d$-dimensional Normal inverse Gaussian process $Y(t)$ can be calculated similar to Theorem \eqref{TSS-Asymptotic-Levy} and is given by
\begin{align}
 \displaystyle   J(x) \sim
\left\{
	\begin{array}{ll}
		\frac{4\delta(d+1)}{\sqrt{2\pi^{d+1}}}|x|^{-(d+1)}  & \mbox{if}\; |x| \rightarrow 0+ \\
		\frac{4\delta (\lambda/2)^{2d}}{\sqrt{2\pi^d}}|x|^{-d/2}e^{-\frac{\lambda}{\sqrt{2}}|x|} & \mbox{if}\; |x| \rightarrow \infty.
	\end{array}
\right.
\end{align}

\noindent Next, we compute the asymptotic behavior of Green function using the approach mentioned in Sikic et al. (2006).

\begin{proposition}
For $\delta, \lambda>0$, we have
\begin{enumerate}
\item $G(x)\sim\frac{1}{2^{\frac{3}{2}}\pi^{\frac{d}{2}}\delta}\frac{\Gamma\left(\frac{d-1}{2}\right)}{\Gamma\left(\frac{1}{2}\right)}\lvert{x}\rvert^{1-d},\;\lvert{x}\rvert\rightarrow{0^+},$
\item $G(x)\sim\frac{\lambda}{4\pi^{\frac{d}{2}}\delta}{\Gamma\left(\frac{d-2}{2}\right)}\lvert{x}\rvert^{2-d},\; \lvert{x}\rvert\rightarrow{\infty}.$
\end{enumerate}
\end{proposition}
\begin{proof}
\noindent\begin{enumerate}
\item As $s\rightarrow \infty$, $\phi(s)\sim\delta\sqrt{2s}$, then it follows directly from Theorem 3.1 of Rao et al. (2006) that 
\begin{equation*}
G(x)\sim\frac{1}{2^{\frac{3}{2}}\pi^{\frac{d}{2}}\delta}\frac{\Gamma\left(\frac{d-1}{2}\right)}{\Gamma\left(\frac{1}{2}\right)}\lvert{x}\rvert^{1-d},\; \lvert{x}\rvert\rightarrow{0^+}.
\end{equation*}
\item  As $s\rightarrow{0^+}$, $\phi(s)\sim\frac{\delta}{\lambda}s$, then using Theorem 3.3 of Rao et al. (2006), it follows 
\begin{equation*}
G(x)\sim\frac{\lambda}{4\pi^{\frac{d}{2}}\delta}\Gamma\left(\frac{d-2}{2}\right)\lvert{x}\rvert^{2-d}, \; \lvert{x}\rvert\rightarrow{\infty}.
\end{equation*}
\end{enumerate}
\end{proof}

\section{Conclusions}
In this paper, we mainly computed the asymptotic behavior of potential density and L\'evy density associated with the tempered stable subordinator and inverse Gaussian subordinator and also the asymptotic behavior of Green function and L\'evy density associated with the Brownian motion timed changed by tempered stable subordinator and inverse Gaussian process. 
\vspace{0.2cm}

\noindent {\bf Acknowledgments:}   
A. Kumar is supported by Science and Engineering Research Board (SERB), India project MTR/2019/000286. This work is part of H. Verma MS thesis work.


\addcontentsline{toc}{section}{Bibliography}
\begin{thebibliography}{11}  

\bibitem{Abramowitz1972} Abramowitz, M., Stegun, I.A. (Eds.), 1992. Handbook of Mathematical Functions with Formulas, Graphs and Mathematical Tables. Dover, New York.

\bibitem{Applebaum2009} Applebaum, D., 2009.
L\'evy Processes and Stochastic Calculus
(second ed.). Cambridge University Press, Cambridge.


\bibitem{Armitage2001} Armitage, D.H., Gardiner, S.J., 2001. Classical Potential Theory. Springer-Verlag, London.


\bibitem{Barndorff-Nielsen1998} Barndorff-Nielsen, O.E., 1998.
Processes of normal inverse Gaussian type.
Finance Stochast. 2, 41-68.

\bibitem{Barndorff-Nielsen2001} Barndorff-Nielsen O.E., Shephard N., 2001.
Normal modified stable processes. Theory Probab. Math. Statist. 65, 1-19.


\bibitem{Bingham1987} Bingham, N.H., Goldie, C.M., Teugels, J.L., 1987. Regular Variation. Cambridge University Press, Cambridge. 

\bibitem{Doob2001} Doob, J.L., 2001. Classical Potential Theory and its Probabilistic Counterpart. Springer-Verlag, Berlin Heidelberg.

\bibitem{Gupta2018} Gupta, N., Kumar, A. and Leonenko, N.N., 2020. Tempered Fractional Poisson Processes and Fractional Equations with $Z$-Transform. Stoch. Anal. Appl. To Appear.

\bibitem{Helms1972} Helms, L.L., 2009. Potential Theory. Springer-Verlag,  London.


\bibitem{Heyde2005} Heyde, C.C., Leonenko, N.N., 2005. Student processes. Adv. Appl. Probab. 37, 342-365.

\bibitem{Fienberg1979} Jorgensen, B., 1982. 
Statistical Properties of the Generalized Inverse Gaussian Distribution. Springer-Verlag, New York.

\bibitem{Kum2015}
Kumar, A., Vellaisamy, P., 2015. Inverse tempered stable subordinators. Statist. Probab. Lett. 103, 134-141.

\bibitem{Kozubowski2006} Kozubowski, T.J., Meerschaert, M.M., Podgorski, K., 2006. Fractional Laplace motion. Adv. Appl. Prob. 38, 451-464.

\bibitem{Laskin2003} Laskin, N., 2003. Fractional Poisson process.  Commun.  Nonlinear Sci.  Numer.  Simul.  8, 201-213.


\bibitem{Liggett2010} Liggett, T.M., 2010. Continuous Time Markov Processes: An Introduction. American Mathematical Society.

\bibitem{Madan1998} Madan, D.B., Carr, P., Chang, E.C., 1998. The variance gamma process and option pricing. Eur. Fin. Rev. 2, 74-105.
   



\bibitem{Mandelbrot1997} Mandelbrot, B.B., Fisher, A., Calvet, L., 1997. A multifractal model of asset returns. Cowles Foundation discussion paper no. 1164.


\bibitem{Mainardi2004} Mainardi, F., Gorenflo, R., Scalas, E., 2004. A fractional generalization of the Poisson processes. Vietnam J. Math. 32, 53-64.

\bibitem{Meer2011} Meerschaert, M.M., Nane, E., Vellaisamy, P., 2011. The fractional Poisson process and the inverse stable subordinator. Electron. J. Probab., 16, 1600-1620.

\bibitem{Orsingher2012} Orsingher, E., Polito, F., 2012. The space-fractional Poisson process. Statist. Probab. Lett. 82, 852-858.


\bibitem{Alexa1999} Poularikas, A.D., 1998. 
Handbook of Formulas and Tables for Signal Processing. Springer Science \& Business Media.

\bibitem{Rao2006} Rao, M., Song, R., Vondracek, Z., 2006. Green function estimates and Harnack inequality for subordinate Brownian motions. Potential  Anal. 25, 1-27.



\bibitem{Rosinski2007} Rosi\'nski, J., 2007. Tempering stable processes. Stochastic Process Appl. 117, 677-707.


\bibitem{Joel1999} Schiff, J.L., 1999. The Laplace Transform: Theory and Applications. Springer-Verlag New York.


\bibitem{Sikic2006} Sikic, H., Song, R., Vondracek, Z., 2006. Potential Theory of Geometric Stable Processes. Probab. Theory Relat. Fields. 135, 547-575.



   




    


 






  

\end{thebibliography}
\end{document}